\RequirePackage{fix-cm}
\documentclass[smallextended]{svjour3}       
\smartqed  
\usepackage{graphicx}
\usepackage{etex,bm}
\usepackage{amsmath,amssymb,mathrsfs,mathtools}
\usepackage{amsfonts}
\usepackage[margin=1in]{geometry}
\usepackage{multirow,booktabs}
\usepackage{empheq}
\usepackage{array}
\usepackage{subcaption}
\usepackage[colorlinks,citecolor=red,urlcolor=blue,bookmarks=false,hypertexnames=true]{hyperref}

\newtheorem{cor}{Corollary}
\begin{document}
\title{Chance Constrained Robust Portfolio Optimization when the Perturbations Follow Normal and Exponential Distributions
}


\author{Pulak Swain         \and
        Akshay Kumar Ojha 
}


\author{Pulak Swain \and Akshay Kumar Ojha }
\institute{Pulak Swain \at
	School of Basic Sciences, Indian Institute of Technology Bhubaneswar \\
	\email{ps28@iitbbs.ac.in}           
	\and
	Akshay Kumar Ojha \at
	School of Basic Sciences, Indian Institute of Technology Bhubaneswar
}

\date{Received: date / Accepted: date}

\maketitle

\begin{abstract}
In this paper, we consider the chance constrained based uncertain portfolio optimization problem in which the uncertain parameters are stochastic in nature. The primary goal of the work is to formulate the uncertain problem into a deterministic model to study its robust counterpart, which will be helpful for solving such types of uncertain problems. In the present study, we assume that the uncertainty occurs in the expected asset returns, accordingly we derive the corresponding robust counterparts for the cases when the perturbations follow normal and exponential distributions. The obtained robust counterparts are computationally tractable. So our study can be used to find out the deterministic robust counterparts of any quadratic programming problems with uncertain constraints. In the end, we solve an Indian stock market problem for the nominal case as well as the cases for normally and exponentially distributed perturbations, and the results are analyzed. 
\end{abstract}

\section{Introduction} \label{sec1}
Portfolio optimization deals with the problem of allocating total wealth of an investor among different assets. This was first introduced by Harry Markowitz in 1952 \cite{Markowitz J 1952} and afterwards it has been studied world wide \cite{Markowitz 1959}. The two important components of this problem are return and risk. The investor aims to draw a balance between the low risk and high return of the portfolio. The tradeoff between risk and return of a portfolio can be shown in terms of a curve known as the efficient frontier. It is the set of
optimal portfolios that offer the highest expected return for a defined level of risk
or the lowest risk for a given level of expected return. Several portfolio models such as mean-variance, mean-semivariance, value at risk, etc. have been used by several researchers \cite{Estrada 2002,Estrada 2007}. Markowitz's mean-variance model uses expected returns and the covariance of returns of the individual assets as the input parameters.\\
For the formulation of the Markowitz mean-variance portfolio model, let us assume that there are $n$ number of assets with the return of an arbitrary $i^{th}$ asset over a time period $t$ be given by $r_{it}$ ($t=1, 2, \dots, T$). The expected return of $i^{th}$ asset be $\mu_i$ and the covariance of return between $i^{th}$ and $j^{th}$ assets be $\sigma_{ij}$ which are defined as,
\begin{align*}
\begin{aligned}
	&\mu_i=\dfrac{1}{T} \sum_{i=1}^{T} r_{it}, \quad \sigma_{ij}=\dfrac{1}{T} \sum_{i=1}^{T} (r_{it}-\mu_i) (r_{jt}-\mu_j)
\end{aligned}
\end{align*}
Let \begin{align*}\bm{\mu}=\begin{bmatrix} 
	\mu_1\\\mu_2\\ \vdots \\\mu_n
\end{bmatrix} \text{ and } \bm{\Sigma}=\begin{bmatrix} 
\sigma_{11}  && \sigma_{12}  && \dots && \sigma_{1n}  \\
\sigma_{21} && \sigma_{22} && \dots  &&\sigma_{2n} \\
\vdots && \vdots && \ddots && \vdots\\
\sigma_{n1} && \sigma_{n2} && \dots && \sigma_{nn}\\
\end{bmatrix}
\end{align*} be respectively the expected return vector and the covariance return matrix of the $n$ assets.\\
Now suppose \begin{align*}\bm{x}=\begin{bmatrix} 
		x_1\\x_2\\ \vdots \\x_n
	\end{bmatrix}\end{align*} be the allocated weight vector of the portfolio, where $x_i$ be the weight allocated for the $i^{th}$ asset. Then the expected portfolio return ($\mu_p$) and the variance of portfolio return ($\sigma_p$) are respectively calculated as,
\begin{align*}
	\begin{aligned}
		&\mu_p=\sum_{i=1}^{n} \mu_i x_i=\bm{\mu}^\top \bm x, \quad \sigma_{p}^2=\sum_{i=1}^{n} \sum_{j=1}^{n} \sigma_{ij} x_i x_j=\bm x^\top \bm{\Sigma} \bm x
	\end{aligned}
\end{align*}
Assume that the investor wants to achieve a target portfolio return of $\tau$ by the asset allocation. The Markowitz portfolio model for this problem is given by,
\begin{align*}
	\begin{aligned}	
		&\min_{\bm x} && \frac{1}{2} \bm x^\top \bm{\Sigma} \bm x \\
		&\textrm{s.t.:} && \bm{\mu}^\top \bm x \geq \tau, \quad	\bm{e}^\top \bm x=1, \quad \bm x \geq \bm 0
	\end{aligned}
\end{align*}
where $\bm e$ is the vector of $1$'s and $\bm 0$ is the vector of $0$'s.\\
Here the main objective of the model is to minimize the portfolio variance, subject to the conditions that (i) the expected portfolio return must be equal to or exceed the target portfolio return, (ii) the sum of allocated asset weights is equal to $1$, (iii) the allocated weights must be non-negative.\\
In the literature, some researchers have also inserted some additional constraints, such as cardinality and bound constraints in the Markowitz model, which restricts the number of nonzero assets and the bound of asset weights in the portfolio \cite{Meghwani 2017}.\\
In practice, it is difficult to get the exact expected returns and the covariance of returns due to several economic factors involved in this process. So there is always the possibility of the presence of an uncertainty factor in the input parameters, and if ignored, that can affect the optimal solution leading to a wrong asset allocation. In the field of optimization, several approaches are there to counter uncertain problems, such as fuzzy programming, dynamic programming, sensitivity analysis, etc. But in the last two decades, the robust optimization approach is widely used due to its ability to give the "completely immunized against uncertainty" solution. The robust optimization approach was first studied by El Ghaoui and Lebret \cite{Ghaoui 1997} in their work based on the uncertain least-squares problems. El Ghaoui et al. \cite{Ghaoui 1998} studied the robust semidefinite problems, and they provided sufficient conditions for a robust solution to exist as semidefinite programming. Ben-Tal and Nemirovski \cite{Ben-Tal 1999} showed that the robust counterpart of uncertain linear programming with ellipsoidal uncertainty is computationally tractable. Again Ben-Tal and Nemirovski \cite{Ben-Tal 2000} used robust optimization methodology for getting solutions for 90 linear programming problems from the NETLIB collection. They wanted to check whether the robust optimization approach affects the optimality of the solutions, and their results showed that the robust solutions nearly lose nothing in optimality. Afterwards, the robust optimization approach became very popular and is widely being used in many disciplines of science and engineering to tackle uncertainty \cite{Bertsimas 2011}. Calafiore and  El Ghaoui \cite{Calafiore 2006} studied the robust optimization approach for chance constrained linear programming problems. They analyzed the case when the probability distribution of the data is not completely known, but is only known to belong to a given class of distributions. Nemirovski \cite{Nemirovski 2012} discussed several simulation-based and simulation-free computationally tractable approximations of chance constrained convex programs, in particular linear, conic quadratic and semidefinite programming problems. The optimal robust optimization approximation for chance constrained optimization problems was studied by Li and Li \cite{Li 2015}, in which they analyzed the relationship between uncertainty set size and solution reliability.\\
The evolving robust optimization approach opened new doors for the uncertain portfolio allocation problems, and it has been handy for the researchers. Goldfarb and Iyenger \cite{Goldfarb 2003} first showed how to formulate and solve robust portfolio selection problems. They introduced some uncertainty structures for the market parameters and reformulated the robust portfolio problems corresponding to those uncertainty structures into second order cone programming. T\"{u}t\"{u}nc\"{u} and Koenig \cite{Tutuncu 2004} studied the case when the uncertainty in the expected return vector and the covariance return matrix are defined in terms of lower and upper bounds. The robust optimization approaches to multiperiod portfolio selection problems were studied by Bertsimas and Pachamanova \cite{Bertsimas 2008}. Their robust multiperiod portfolio optimization formulations were linear and computationally efficient. Fliege and Werner \cite{Fliege 2014} studied robust multiobjective optimization and its application on uncertain mean-variance portfolio problems. Kim et al. \cite{Kim 2018} did a comprehensive analysis of robust portfolio performance in the U.S. market from 1980 to 2014 and showed the advantage of robust portfolio optimization for controlling uncertainty. Ismail and
Pham \cite{Ismail 2019} studied a robust continuous-time Markowitz model with an uncertain covariance matrix of risky assets. They also compared the performance of Sharpe ratios for a robust investor and an investor with a misspecified model. Sehgal and Mehra \cite{Sehgal 2019} proposed robust portfolio optimization models for reward–risk ratios utilizing Omega, semi-mean absolute deviation ratio, and weighted stable tail adjusted return ratio, and they evaluated the performance of these models on the listed stocks of FTSE 100, Nikkei 225, S\&P 500, and S\&P BSE 500.\\
The paper's main contribution is as follows: Since the chance constrained problems are generally challenging to solve and are intractable, here we derive the deterministic form of the robust counterparts. Our study is based on cases when perturbations follow normal and exponential distributions. The results are then implemented in an Indian stock market problem to determine the robust optimal allocation. In addition, risk-return efficient frontiers are drawn for the nominal and the two above-mentioned distribution cases. A dissimilarity analysis among the results of the three models is also done in the end.\\
The rest of the paper is organized as follows: Section \ref{sec2} presents some definitions and basic concepts regarding robust optimization and the chance constrained based portfolio optimization problems. Then the deterministic robust counterparts of the uncertain portfolio problem for the normally and exponentially distributed perturbations are derived in Section \ref{sec3}. A numerical example of a stock market problem is given in Section \ref{sec4}. Finally, some concluding remarks are provided in Section \ref{sec5}.
\section{Preliminaries} \label{sec2}
\subsection{Uncertain Optimization Problems and the Robust Counterpart}
An uncertain optimization problem is defined as the optimization problem in which the constant coefficients are not known exactly; rather we only know that those values perturb around some nominal values. Mathematically this can be represented as,
\begin{align} \label{eq0}
		\begin{aligned}	
			&\min_{\bm x} && f(\bm{x},\bm{u}) \\
			&\textrm{s.t.:} && c(\bm{x},\bm{u}) \leq 0, \quad \forall \bm u = \bm{u}^{(0)}+\displaystyle{\sum_{j=1}^{n}}\zeta_j \bm{u}^{(j)}  \in \mathscr{U} 
		\end{aligned}
\end{align}
where $\bm x$ is the vector of decision variables, $\bm u$ is the vector of uncertain variables, $\bm{u}^{(0)}$ is the nominal vector of the uncertain parameters and $\bm{u}^{(j)}$ are the basic shifts in the uncertain parameters.\\
The uncertain set $\mathscr{U}$ in problem \eqref{eq0} is represented by the affine parameterization of perturbations $\zeta_j$'s. To understand this suppose our problem has $n$ uncertain parameters $u_1, u_2, \dots, u_n$, whose nominal values are $u_1^{(0)}, u_2^{(0)}, \dots, u_n^{(0)}$. Let the values of $u_1, u_2, \dots, u_n$ perturb upto $\delta_1, \delta_2, \dots, \delta_n$ respectively. Then the uncertain set $\mathscr{U}$ can be written as either of the two following forms:
\begin{align*}
	\begin{aligned}	
		\mathscr{U}=\left\{ \begin{bmatrix} 
			u_1\\u_2\\ \vdots \\u_n
		\end{bmatrix} : \begin{bmatrix} 
			u_1^{(0)}\\u_2^{(0)}\\ \vdots \\u_n^{(0)}
		\end{bmatrix}+ \zeta_1 \begin{bmatrix} 
			\delta_1\\0\\ \vdots \\0 \end{bmatrix} + \zeta_2 \begin{bmatrix} 
			0\\ \delta_2\\ \vdots \\0 \end{bmatrix} + \dots +\zeta_n \begin{bmatrix} 
			0\\0\\ \vdots \\ \delta_n \end{bmatrix},\zeta_j \in [-1, 1],\ j=1,2, \dots, n \right\}
	\end{aligned}
\end{align*}
or \begin{align*}
	\begin{aligned}	
		\mathscr{U}=\left\{ \begin{bmatrix} 
			u_1\\u_2\\ \vdots \\u_n
		\end{bmatrix} : \begin{bmatrix} 
		u_1^{(0)}\\u_2^{(0)}\\ \vdots \\u_n^{(0)}
	\end{bmatrix}+ \zeta_1 \begin{bmatrix} 
	1\\0\\ \vdots \\0 \end{bmatrix} + \zeta_2 \begin{bmatrix} 
	0\\ 1\\ \vdots \\0 \end{bmatrix} + \dots +\zeta_n \begin{bmatrix} 
	0\\0\\ \vdots \\ 1 \end{bmatrix},\ \zeta_j \in [-\delta_j, \delta_j],\ j=1,2, \dots, n \right\}
\end{aligned}
\end{align*}
Next let us go through some important definitions related to the robust optimization.
\begin{definition} [Robust Feasible Solution] \label{def1} \cite{Ben-Tal 2009}
A vector $\bm x$ is said to be the robust feasible solution of the uncertain problem \eqref{eq0} if it satisfies the uncertain constraints for all realizations of the uncertain set $\mathscr{U}$, that is, if $\bm x$ satisfies
\begin{align*}
\begin{aligned}	
c(\bm{x},\bm{u}) \leq 0, \quad \forall \bm u  \in \mathscr{U} 
\end{aligned}
\end{align*}
\end{definition}
\begin{definition} [Robust Value] \label{def2} \cite{Ben-Tal 2009}
	Given a candidate solution $\bm{x}$, the robust value $\widehat{f}(\bm{x})$ of the objective in problem \eqref{eq0} is the largest value of $f(\bm{x},\bm{u})$ over all realizations of the data from the uncertain set, that is
	\begin{align*}
	\begin{aligned}	
	\widehat{f}(\bm{x})= \sup_{\bm u  \in \mathscr{U}} f(\bm{x},\bm{u})
	\end{aligned}
	\end{align*}
\end{definition}
It is to be noted that if the problem \eqref{eq0} were a maximization problem, then the robust value would have been $\widehat{f}(\bm{x})= \inf_{\bm u  \in \mathscr{U}} f(\bm{x},\bm{u})$.
\begin{definition} [Robust Counterpart] \label{def3} \cite{Ben-Tal 2009}
The robust counterpart (RC) of the uncertain problem \eqref{eq0} is the optimization problem 
\begin{align*}	
\begin{aligned} 
& \min_{\bm x} && \left \{ \sup_{\bm u \in \mathscr{U}}	f(\bm x, \bm u) \right \}\\
&\text{s.t.:} && c(\bm x, \bm u) \leq \bm 0, \quad \forall \bm u  \in \mathscr{U}
\end{aligned}
\end{align*}
of minimizing the robust value of the objective over all the robust feasible solutions to the uncertain problem.
\end{definition}
\begin{definition} [Robust Optimal Solution] \label{def4} \cite{Ben-Tal 2009}
The solution of the RC problem is said to be the robust optimal solution of the uncertain problem \eqref{eq0}.
\end{definition}
\subsection{Probability Chance Constraint}
Consider the uncertain constraint given by,
\begin{align}
	\begin{aligned}	\label{eq1}
		c(\bm x, \bm u) \geq 0, \quad \bm{u}= \bm{u}^{(0)}+\displaystyle{\sum_{j=1}^{n}}\zeta_j \bm{u}^{(j)}
	\end{aligned}
\end{align}
In the ideal situation, we would like to find a solution $\bm x$ which satisfies the constraint (\ref{eq1}) for all realizations of the uncertain set. But when the data perturbation is stochastic in nature, this is not always possible to find such a solution. So we look for a candidate solution that satisfies (\ref{eq1}) for "nearly all" realizations of $\bm{\zeta}$. So the probability chance constraint to (\ref{eq1}) is introduced which is given by \cite{Pagnoncelli 2015}, 
\begin{align}
\begin{aligned}		\label{eq2}
{Prob}_{\bm{\zeta}\sim P} \left \{\bm{\zeta}: c(\bm x, \bm u) \geq 0, \quad \bm{u}= \bm{u}^{(0)}+\displaystyle{\sum_{j=1}^{n}}\zeta_j \bm{u}^{(j)} \right \} \geq \beta
\end{aligned}
\end{align}
Eq. (\ref{eq2}) indicates that the constraint (\ref{eq1}) would be satisfied with a probability of at least $\beta$, where $\beta$ is the level of confidence and its value lies in the interval $(0,1)$. We aim to make $\beta$ as close as possible to 1. In this chance constrained problem, the perturbation vector $\bm{\zeta}$ is treated as the random variable with the probability distribution $P$.\\
One of the most significant drawbacks of these chance constraint problems is the tractability issues of the robust counterparts. The problem becomes NP-hard even for the cases when $P$ is simple.
\subsection{Formulation of Uncertain Portfolio Problem with Chance Constraint}
It has been found from the studies that the uncertainty in the covariance of asset returns does not affect the optimal solution as much as the uncertainty in expected asset returns does \cite{Pulak 2021}. So let us assume that the uncertainty occurs only in the assets' expected returns.\\ 
Then the uncertain portfolio optimization model \cite{Fabozzi 2007} is given by,
\begin{align}
\begin{aligned}	\label{eq3}
	&\min_{\bm x} && \frac{1}{2} \bm x^\top \bm{\Sigma} \bm x \\
	&\textrm{s.t.:} && \bm{\mu}^\top \bm x \geq \tau, \quad \bm{\mu} \equiv \bm{\mu}(\bm{\zeta})= \bm{\mu}^{(0)}+\displaystyle{\sum_{j=1}^{n}}\zeta_j \bm{\mu}^{(j)} \in \mathscr{U}_{\bm{\mu}}, \\
	& &&	\bm{e}^\top \bm x=1, \quad \bm x \geq \bm 0
\end{aligned}
\end{align}
where the expected return vector $\bm{\mu}$ belongs to some uncertain set $\mathscr{U}_{\bm{\mu}}$ and $\bm{\zeta}=[\zeta_1 \ \zeta_2 \ \dots \ \zeta_n]^\top$ be the perturbation vector associated with $\bm{\mu}$.\\
The corresponding nominal model of the problem \eqref{eq3} is given by,
\begin{align}
\begin{aligned}	\label{eq3a}
&\min_{\bm x} && \frac{1}{2} \bm x^\top \bm{\Sigma} \bm x \\
&\textrm{s.t.:} && [\bm{\mu^{(0)}}]^\top \bm x \geq \tau, \\
& &&	\bm{e}^\top \bm x=1, \quad \bm x \geq \bm 0
\end{aligned}
\end{align}
When the perturbations are stochastic and follow some probability distributions, the problem \eqref{eq3} can be written with chance constraint as,
\begin{align}
\begin{aligned}	\label{eq4}
	&\min_{x_i} && \frac{1}{2} \bm x^\top \bm{\Sigma} \bm x \\
	&\textrm{s.t.:} && {Prob}_{{\bm{\zeta}}\sim P} \left \{ [\bm{\mu}(\bm{\zeta})]^\top \bm{x} \geq \tau \right \} \geq \beta,\\
	& && \bm{e}^\top \bm x=1, \quad \bm x \geq \bm 0.
\end{aligned}
\end{align}
The chance constraint in problem (\ref{eq4}) can be written as,
\begin{align}
\begin{aligned}	\label{eq5}
& {Prob}_{\bm{\zeta}\sim P} \left \{[\bm{\mu}(\bm{\zeta})]^\top \bm x <\tau \right \} \leq 1-\beta
\end{aligned}
\end{align}
The aim is to find the robust counterparts by considering the perturbation vector $\bm{\zeta}$ as the random variable with different probability distributions $P$.
\section{Deterministic Robust Counterparts for the Perturbations with Several Known Probability Distributions} \label{sec3}
In general, the chance constrained problems are difficult to solve due to the intractability of the robust counterparts in most cases. So in this section, we aim to derive the deterministic form of robust counterparts when the probability distribution of $\bm{\zeta}$ is known.\\
Since the objective function in our portfolio problem is certain, the value of the objective function becomes the robust value, and thus to find out the robust counterpart of the problem \eqref{eq4} we need to replace the chance constraint with a deterministic constraint which would be feasible. Going forward, we must first write the chance constraint in component form for convenience.\\
 Now the uncertain vector $\bm{\mu}(\bm{\zeta})$ can be written as,
\begin{align*}
\begin{aligned}
& \bm{\mu}(\bm{\zeta})&&=\begin{bmatrix} 
\mu_1^{(0)}\\ \mu_2^{(0)}\\ \vdots \\ \mu_n^{(0)}
\end{bmatrix}+ \zeta_1 \begin{bmatrix} 
\mu_1^{(1)}\\0\\ \vdots \\0 \end{bmatrix} + \zeta_2 \begin{bmatrix} 
0\\ \mu_2^{(2)}\\ \vdots \\0 \end{bmatrix} + \dots +\zeta_n \begin{bmatrix} 
0\\0\\ \vdots \\ \mu_n^{(n)} \end{bmatrix}\\
& &&= \begin{bmatrix} 
\mu_1^{(0)}+\zeta_1 \mu_1^{(1)}\\ \mu_2^{(0)}+\zeta_2 \mu_2^{(2)}\\ \vdots \\ \mu_n^{(0)}+\zeta_n \mu_n^{(n)} \end{bmatrix}
\end{aligned}
\end{align*}
So the left-hand side of the uncertain constraint in component form reduces to,
\begin{align*}
\begin{aligned}
& [\bm{\mu}(\bm{\zeta})]^\top \bm x&&= \begin{bmatrix} 
\mu_1^{(0)}+\zeta_1 \mu_1^{(1)}\\ \mu_2^{(0)}+\zeta_2 \mu_2^{(2)}\\ \vdots \\ \mu_n^{(0)}+\zeta_n \mu_n^{(n)} \end{bmatrix}^\top \begin{bmatrix} 
x_1\\ x_2 \\ \vdots \\ x_n \end{bmatrix}\\
& &&=\displaystyle{\sum_{j=1}^{n}} \mu_j^{(0)} x_j+ \displaystyle{\sum_{j=1}^{n}}(\mu_j^{(j)} x_j)\zeta_j
\end{aligned}
\end{align*}
Then we can rewrite the chance constraint (\ref{eq5}) in the form of $\mu_j$ and $\zeta_j$ components as,
\begin{align}
	\begin{aligned}	\label{eq6}
		&{Prob}_{\bm{\zeta_j}} \left \{\displaystyle{\sum_{j=1}^{n}} \mu_j^{(0)} x_j+ \displaystyle{\sum_{j=1}^{n}}(\mu_j^{(j)} x_j)\zeta_j <\tau \right \} \leq 1-\beta\\
		\implies & {Prob}_{\bm{\zeta_j}} \left \{ \displaystyle{\sum_{j=1}^{n}}(\mu_j^{(j)} x_j)\zeta_j <\tau-\displaystyle{\sum_{j=1}^{n}} \mu_j^{(0)} x_j \right \} \leq 1-\beta
	\end{aligned}
\end{align}
The chance constraint (\ref{eq6}) can be written as,
\begin{align}
	\begin{aligned}	\label{eq7}
		& F_{Y} (\tau-\displaystyle{\sum_{j=1}^{n}} \mu_j^{(0)} x_j) \leq 1-\beta;
	\end{aligned}
\end{align}
where $F_{Y} (\tau-\displaystyle{\sum_{j=1}^{n}} \mu_j^{(0)} x_j)$ is the cumulative distribution function (CDF) of the random variable $Y=\sum_{j=1}^{n} (\mu_j^{(j)} x_j)\zeta_j$.\\
Therefore to obtain the deterministic form of our chance constrained problem, we require the CDF of the random variable $Y$, and for that, we need to find out the probability distribution of $Y$. Next, we discuss two cases-- when the perturbations $\zeta_j$'s follow (i) normal and (ii) exponential distributions.
\subsection{When the Perturbations Follow Normal Distribution}
Suppose the perturbations $\zeta_1, \zeta_2, \dots, \zeta_n$ are independent and identically distributed normal random variables, having means $m_1, m_2, \dots, m_n$ and standard deviations $s_1, s_2, \dots, s_n$ respectively.\\
 The probability density function (PDF) of the random variable $\zeta_j$ is given by,
 \begin{align*}
 f(\zeta_j)=\dfrac{1}{s_j \sqrt{2 \pi}} e^{\frac{-(\zeta_j)^2}{2 (s_j)^2}}
 \end{align*}
\begin{figure}[h!]
	\centering
	\subfloat[PDF of $\zeta_j$ with $\text{mean}=0$ and $\text{standard deviation}=1$ \label{fig0a}]{\includegraphics[height=5cm,width=7.5cm]{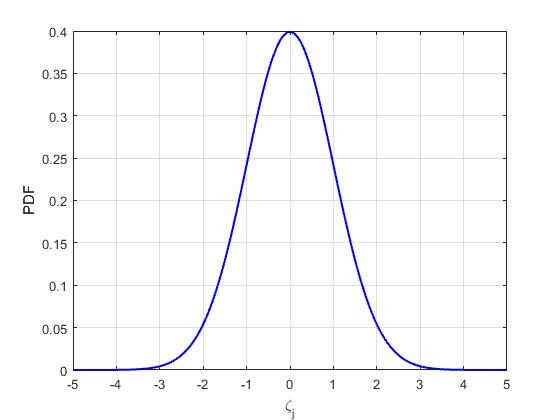}}\qquad
	\subfloat[CDF of $\zeta_j$ with $\text{mean}=0$ and $\text{standard deviation}=1$\label{fig0b}]{\includegraphics[height=5cm,width=7.5cm]{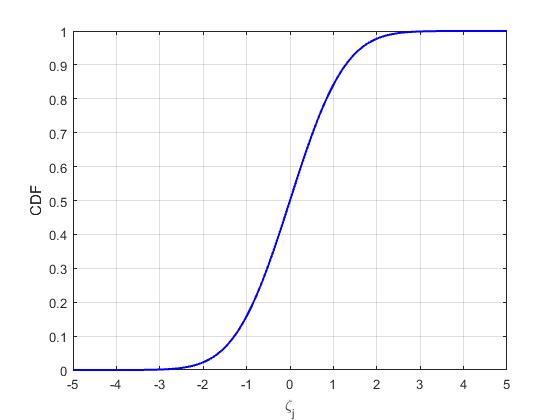}}
	\caption{Probability Density Function (PDF) and Cumulative Distribution Function (CDF) of the normally distributed perturbation $\zeta_j$}
		\label{fig0}
\end{figure}
 We want to find out the deterministic form of our chance constrained problem.\\
\begin{lemma} \label{lemma1}
	If $\zeta_j,\ j=1, 2, \dots, n$	are independent and identically distributed normal random variables, with mean $m_j$ and standard deviation $s_j$, then the random variable $Y=\sum_{j=1}^{n} (\mu_j^{(j)} x_j)\zeta_j$ follows the normal distribution and it has mean $m=\sum_{j=1}^{n} (\mu_j^{(j)} x_j)m_j$ and standard deviation $s=\sqrt{\sum_{j=1}^{n}(s_j \mu_j^{(j)} x_j)^2}$.
\end{lemma}
\begin{proof}
The moment generating function (MGF) of normally distributed random perturbation $\zeta_j$ with mean $m_j$ and standard deviation $s_j$ is given by,
\begin{align}
\begin{aligned} \label{eq8}	
M_{\zeta_j}(z)= E[e^{z\zeta_j}]= e^{m_jz+\frac{1}{2}s_j^2 z^2}
\end{aligned}
\end{align}
In the random variable $(\mu_j^{(j)} x_j)\zeta_j$ the coefficient of $\zeta_j$, that is, $\mu_j^{(j)}x_j$ is free from any probability distribution.\\
This implies,
\begin{align*}
\begin{aligned}	
M_{(\mu_j^{(j)}x_j)\zeta_j}(z)= E[e^{z(\mu_j^{(j)}x_j\zeta_j)}]= e^{(\mu_j^{(j)} x_j m_j)z+\frac{1}{2} (\mu_j^{(j)}x_js_j)^2 z^2}
\end{aligned}
\end{align*}
Now MGF of the random variable $Y=\sum_{j=1}^{n} (\mu_j^{(j)} x_j)\zeta_j$ is given by,
\begin{align}
\begin{aligned}	\label{eq9}
& M_{Y}(z)&& = M_{\sum_{j=1}^{n} (\mu_j^{(j)} x_j)\zeta_j}(z)\\
& &&= \prod_{j=1}^{n}  M_{(\mu_j^{(j)} x_j)\zeta_j}(z)\\
& &&= \prod_{j=1}^{n}  e^{(\mu_j^{(j)} x_j m_j)z+\frac{1}{2} (\mu_j^{(j)}x_js_j)^2 z^2}\\
& &&= e^{\sum_{j=1}^{n}\left [\mu_j^{(j)} x_j m_j z+\frac{1}{2}s_j^2 (\mu_j^{(j)})^2x_j^2z^2\right ]}
\end{aligned}
\end{align}
From the MGF equation \eqref{eq9}, it is clear that the random variable $Y$ is normally distributed with its mean and standard deviation respectively given by,
\begin{align}
\begin{aligned}	\label{eq10}
m=\sum_{j=1}^{n} \mu_j^{(j)} x_j m_j, \text{ and } s= \sqrt{\displaystyle{\sum_{j=1}^{n}}(s_j \mu_j^{(j)} x_j)^2}
\end{aligned}
\end{align}
\qed
\end{proof}
By using Lemma \ref{lemma1} the probability density function of normally distributed random variable $Y$ becomes,
\begin{align}
	\begin{aligned}	\label{eq11}
		f(y)= &\dfrac{1}{s\sqrt{2 \pi}} \cdot e^{\dfrac{-(y-m)^2}{2v}}
	\end{aligned}
\end{align}
where $m$, $s$ and $v$ are respectively the mean, standard deviation and variance of $Y$.\\
Now the chance constraint (\ref{eq7}) can be written as,
\begin{align}
	\begin{aligned}	\label{eq12}
		& \int_{-\infty}^{\tau-\displaystyle{\sum_{j=1}^{n}}\mu_j^{(0)} x_j} f(y) \cdot dy \leq 1- \beta\\
		\implies & \dfrac{1}{s\sqrt{2 \pi}} \int_{-\infty}^{\tau-\displaystyle{\sum_{j=1}^{n}}\mu_j^{(0)} x_j} e^{\dfrac{-(y-m)^2}{2v}} \cdot dy \leq 1- \beta
	\end{aligned}
\end{align}
The value of integration $\displaystyle{\int e^{\dfrac{-(y-m)^2}{2v}} \cdot dy}$ equals to $\dfrac{\sqrt{\pi}\sqrt{v}\cdot erf\left(\dfrac{y-m}{\sqrt{2}\sqrt{v}}\right)}{\sqrt{2}}$,
where $erf(\cdot)$ is the Gauss error function.\\
Therefore the chance constraint (\ref{eq12}) reduces to,
\begin{align}
	\begin{aligned}	\label{eq13}
		& \dfrac{1}{s\sqrt{2 \pi}} \left [ \dfrac{\sqrt{\pi}\sqrt{v}\cdot erf\left(\dfrac{y-m}{\sqrt{2}\sqrt{v}}\right)}{\sqrt{2}} \right ]\Biggr|_{-\infty}^{\tau-\displaystyle{\sum_{j=1}^{n}}\mu_j^{(0)} x_j} \leq 1- \beta\\
		\implies & \dfrac{1}{2} erf \left( \dfrac{y-m}{\sqrt{2}\sqrt{v}} \right)\Biggr|_{-\infty}^{\tau-\displaystyle{\sum_{j=1}^{n}}\mu_j^{(0)} x_j} \leq 1- \beta 
	\end{aligned}
\end{align}
Using the values of $m$ and $v=s^2$ from eq (\ref{eq10}) our chance constraint (\ref{eq13}) becomes,
\begin{align}
	\begin{aligned}	\label{eq14}
	& \dfrac{1}{2} erf \left( \dfrac{y-\displaystyle{\sum_{j=1}^{n}} \mu_j^{(j)} x_j m_j}{\sqrt{2}\sqrt{\displaystyle{\sum_{j=1}^{n}}(s_j \mu_j^{(j)} x_j)^2}} \right)\Biggr|_{-\infty}^{\tau-\displaystyle{\sum_{j=1}^{n}}\mu_j^{(0)} x_j} \leq 1- \beta \\
	\implies & \dfrac{1}{2} \left [ erf \left( \dfrac{\tau-\left (\displaystyle{\sum_{j=1}^{n}}\mu_j^{(0)} x_j+\sum_{j=1}^{n} \mu_j^{(j)} x_j m_j\right )}{\sqrt{2}\sqrt{\displaystyle{\sum_{j=1}^{n}}(s_j \mu_j^{(j)} x_j)^2}} \right) - (-1)\right ] \leq 1- \beta,
	\end{aligned}
\end{align}
as $erf(-\infty)=-1$.\\
Now the chance constraint (\ref{eq14}) can be simplified as,
\begin{align}
	\begin{aligned}	\label{eq15}
		& erf \left [\dfrac{\tau-\left (\displaystyle{\sum_{j=1}^{n}}\mu_j^{(0)} x_j+\sum_{j=1}^{n} \mu_j^{(j)} x_j m_j\right )}{\sqrt{2}\sqrt{\displaystyle{\sum_{j=1}^{n}}(s_j \mu_j^{(j)} x_j)^2}} \right ] \leq 1- 2\beta \\
		\implies & \tau-\left (\displaystyle{\sum_{j=1}^{n}}\mu_j^{(0)} x_j+\sum_{j=1}^{n} \mu_j^{(j)} x_j m_j\right ) \leq \sqrt{2} \cdot erf^{-1}(1-2\beta) \cdot \sqrt{\displaystyle{\sum_{j=1}^{n}}(s_j \mu_j^{(j)} x_j)^2}\\
		\implies & \displaystyle{\sum_{j=1}^{n}}\mu_j^{(0)} x_j+\sum_{j=1}^{n} \mu_j^{(j)} x_j m_j+ \sqrt{2}  \cdot erf^{-1}(1-2\beta) \cdot \sqrt{\displaystyle{\sum_{j=1}^{n}}(s_j \mu_j^{(j)} x_j)^2} \geq \tau
	\end{aligned}
\end{align}
Any solution which satisfies the constraint \eqref{eq15} is a robust feasible solution of the chance constrained problem \eqref{eq4}. In addition, since the considered portfolio model has a certain objective function, the objective function is itself the robust value. Hence to obtain the robust counterpart of the problem \eqref{eq4}, we can replace the chance constraint with the deterministic constraint given in \eqref{eq15}. \\
Therefore the robust counterpart of the chance constrained problem (\ref{eq4}) can be written in a second order conic programming (SOCP) form as,
\begin{align}
	\begin{aligned}	\label{eq16}
		&\min_{x_i} && \frac{1}{2} \sum_{i=1}^{n} \sum_{j=1}^{n}\sigma_{ij} x_i x_j \\
		&\textrm{s.t.:} &&\displaystyle{\sum_{j=1}^{n}}\mu_j^{(0)} x_j+\sum_{j=1}^{n} \mu_j^{(j)} x_j m_j+ \sqrt{2}  \cdot erf^{-1}(1-2\beta) \cdot \sqrt{\displaystyle{\sum_{j=1}^{n}}(s_j \mu_j^{(j)} x_j)^2} \geq \tau,\\
		& && \sum_{i=1}^{n} {x_i}=1, \quad {x_i}\geq 0, \ i=1, 2, \dots, n.
	\end{aligned}
\end{align}
So this can be represented as the following result:
\begin{theorem} \label{theorem1}
Consider the chance constrained portfolio optimization problem (\ref{eq4}). The uncertain perturbations $\zeta_j, \ j=1, 2, \dots, n$ are independent and identically distributed normal random variables, with mean $m_j$ and standard deviation $s_j$. Then the deterministic form of the robust counterpart of the problem (\ref{eq4}) is given by a SOCP problem,
\begin{align*}
\begin{aligned}	
	&\min_{x_i} && \frac{1}{2} \sum_{i=1}^{n} \sum_{j=1}^{n} \sigma_{ij} x_i x_j \\
	&\textrm{s.t.:} &&\displaystyle{\sum_{j=1}^{n}}\mu_j^{(0)} x_j+\sum_{j=1}^{n} \mu_j^{(j)} x_j m_j+ \sqrt{2}  \cdot erf^{-1}(1-2\beta) \cdot \sqrt{\displaystyle{\sum_{j=1}^{n}}(s_j \mu_j^{(j)} x_j)^2} \geq \tau,\\
	& && \sum_{i=1}^{n} {x_i}=1, \quad {x_i}\geq 0, \ i=1, 2, \dots, n.
\end{aligned}
\end{align*}
\end{theorem}
\begin{cor}
In particular, when the uncertain perturbations $\zeta_j$, $j=1, 2, \dots, n$ in the chance constrained portfolio problem \eqref{eq4} are independent and identically distributed normal random variables, each having a zero mean, the deterministic form of the robust counterpart of the problem for $\beta=0.5$ is same as the nominal problem given by,
\begin{align*}
\begin{aligned}	
&\min_{x_i} && \frac{1}{2} \sum_{i=1}^{n} \sum_{j=1}^{n} \sigma_{ij} x_i x_j \\
&\textrm{s.t.:} &&\displaystyle{\sum_{j=1}^{n}}\mu_j^{(0)} x_j \geq \tau,\\
& && \sum_{i=1}^{n} {x_i}=1, \quad {x_i}\geq 0, \ i=1, 2, \dots, n.
\end{aligned}
\end{align*}
\end{cor}
\begin{proof}
When $m_j=0$, for each $j=1, 2, \dots, n$, the robust counterpart problem \eqref{eq16} becomes,
\begin{align} \label{eq16a}
\begin{aligned}	
&\min_{x_i} && \frac{1}{2} \sum_{i=1}^{n} \sum_{j=1}^{n} \sigma_{ij} x_i x_j \\
&\textrm{s.t.:} &&\displaystyle{\sum_{j=1}^{n}}\mu_j^{(0)} x_j+ \sqrt{2}  \cdot erf^{-1}(1-2\beta) \cdot \sqrt{\displaystyle{\sum_{j=1}^{n}}(s_j \mu_j^{(j)} x_j)^2} \geq \tau,\\
& && \sum_{i=1}^{n} {x_i}=1, \quad {x_i}\geq 0, \ i=1, 2, \dots, n.
\end{aligned}
\end{align}
For $\beta=0.5$ the value of $erf^{-1}(1-2\beta)$ is equal to $0$ and thus the robust counterpart problem \eqref{eq16a} reduces to,
\begin{align}
\begin{aligned}	\label{eq16b}
&\min_{x_i} && \frac{1}{2} \sum_{i=1}^{n} \sum_{j=1}^{n} \sigma_{ij} x_i x_j \\
&\textrm{s.t.:} &&\displaystyle{\sum_{j=1}^{n}}\mu_j^{(0)} x_j \geq \tau,\\
& && \sum_{i=1}^{n} {x_i}=1, \quad {x_i}\geq 0, \ i=1, 2, \dots, n,
\end{aligned}
\end{align}
which is nothing but our nominal problem \eqref{eq3a} in component form.
\qed
\end{proof}
 \subsection{When the Perturbations Follow Exponential Distribution}
Now we assume that the uncertain parameters $\zeta_1, \zeta_2, \dots, \zeta_n$ follow exponential distribution.\\
 The joint probability density function of $\zeta_1, \zeta_2, \dots, \zeta_n$ can be given as,
 \begin{align*}
 \begin{aligned}	
 &f(\zeta_1, \zeta_2, \dots, \zeta_n)=\left \{ \begin{array}{ll} \lambda_1 \lambda_2 \dots \lambda_n e^{-\lambda_1 \zeta_1-\lambda_2 \zeta_2-\dots-\lambda_n \zeta_n} &\mbox{ if } \zeta_1, \zeta_2, \dots, \zeta_n > 0\\
 0 &\mbox{ elsewhere } \end{array} \right.
 \end{aligned}
 \end{align*}
 where $E(\zeta_1)=\dfrac{1}{\lambda_1}$, $E(\zeta_2)=\dfrac{1}{\lambda_2}$, $\dots$, $E(\zeta_n)=\dfrac{1}{\lambda_n}$ and $\lambda_1, \lambda_2, \dots, \lambda_n >0$.\\
 \begin{figure}[h!]
 	\centering
 	\subfloat[PDF of $\zeta_j$ with $\text{mean}=1$ \label{fig00a}]{\includegraphics[height=5cm,width=7.5cm]{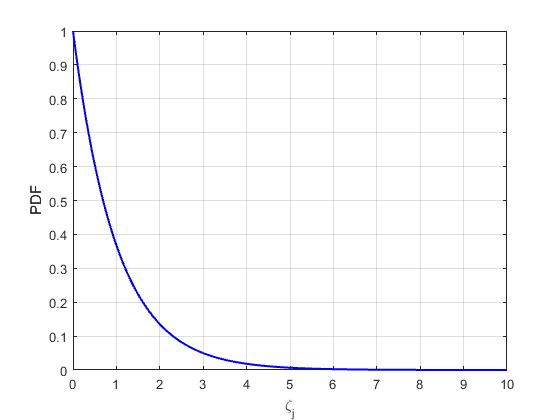}}\qquad
 	\subfloat[CDF of $\zeta_j$ with $\text{mean}=1$ \label{fig00b}]{\includegraphics[height=5cm,width=7.5cm]{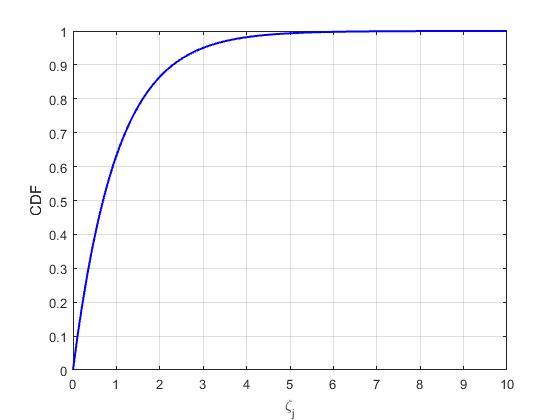}}
 	\caption{PDF and CDF of the exponentially distributed perturbation $\zeta_j$}
 	\label{fig00}
 \end{figure}	
Before proceeding further, we need to find out the probability distribution of $Y=\displaystyle{\sum_{j=1}^{n}}(\mu_j^{(j)} x_j)\zeta_j$, so that we can use it for finding out the deterministic form of the constraint \eqref{eq7}.\\
For this, we use the following result due to Biswal et al. \cite{Biswal 1998}:
\begin{theorem} \label{theorem2} \cite{Biswal 1998}
	If $a_{ij}$, $j=1, 2, \dots, n$ are independent exponential random variables with known means, then for some scalars $x_j,\ j=1, 2, \dots, n$ the probability density function of the random variable $Y_i=\sum_{j=1}^n a_{ij}x_j$ is given by,
	 \begin{align*}
		\begin{aligned}	
			&g_i(y_i)= \displaystyle{\prod_{j=1}^{n}} \lambda_{ij} \left [\sum_{k=1}^{n} \dfrac{(x_k)^{n-2}e^{\frac{-\lambda_{ik}y_i}{x_k}}}{\displaystyle{\displaystyle{\prod_{\substack{l=1 \\ l\neq k}}^{n}}} (x_k \lambda_{il}-x_l \lambda_{ik})} \right ], \quad \text{if } y_i>0
		\end{aligned}
	\end{align*}
where $E(a_{ij})=\dfrac{1}{\lambda_{ij}}$, $\lambda_{ij}>0$, $j=1, 2, \dots, n$.
	\end{theorem}
Thus the probability density function of our random variable $Y=\displaystyle{\sum_{j=1}^{n}}(\mu_j^{(j)} x_j)\zeta_j$ given by,
\begin{align*}
	\begin{aligned}	
		&g(y)=\left \{ \begin{array}{ll} \displaystyle{\prod_{i=1}^{n}} \lambda_{i} \displaystyle{\sum_{k=1}^{n}} \dfrac{(\mu_k^{(k)} x_k)^{n-2}e^{\frac{-\lambda_{k}y}{\mu_k^{(k)} x_k}}}{\displaystyle{\displaystyle{\prod_{\substack{l=1 \\ l\neq k}}^{n}}} (\mu_k^{(k)} x_k \lambda_{l}-\mu_l^{(l)} x_l \lambda_{k})} &\mbox{ if } y > 0\\
			0 &\mbox{ elsewhere } \end{array} \right.
	\end{aligned}
\end{align*}
We can rewrite the chance constraint \eqref{eq7} as,
\begin{align}
	\begin{aligned}	\label{eq17}
		& \int_{0}^{\tau-{\sum_{j=1}^{n}}\mu_j^{(0)} x_j} g(y) \cdot dy \leq 1- \beta
	\end{aligned}
\end{align}
Simplifying the left-hand side of the constraint \eqref{eq17}, we get,
\begin{align*}
	\begin{aligned}
\int_{0}^{\tau-{\sum_{j=1}^{n}}\mu_j^{(0)} x_j} g(y) \cdot dy = & \prod_{i=1}^{n} \lambda_{i} \int_{0}^{\tau-{\sum_{j=1}^{n}}\mu_j^{(0)} x_j} \left [\sum_{k=1}^{n} \dfrac{(\mu_k^{(k)} x_k)^{n-2}e^{\frac{-\lambda_{k}y}{\mu_k^{(k)} x_k}}}{\displaystyle{\prod_{\substack{l=1 \\ l\neq k}}^{n}} (\mu_k^{(k)} x_k \lambda_{l}-\mu_l^{(l)} x_l \lambda_{k})} \right ] \cdot dy\\
= & \prod_{i=1}^{n} \lambda_{i}  {\left [\sum_{k=1}^{n} \dfrac{(\mu_k^{(k)} x_k)^{n-1}e^{\frac{-\lambda_{k}y}{\mu_k^{(k)} x_k}}}{-\lambda_k\displaystyle{\prod_{\substack{l=1 \\ l\neq k}}^{n}} (\mu_k^{(k)} x_k \lambda_{l}-\mu_l^{(l)} x_l \lambda_{k})}  \right ]} \Biggr|_{0}^{\tau-{\sum_{j=1}^{n}}\mu_j^{(0)} x_j}\\
= & \prod_{i=1}^{n} \lambda_{i}  \left [\sum_{k=1}^{n} \dfrac{(\mu_k^{(k)} x_k)^{n-1}\left(e^{\frac{-\lambda_{k}(\tau-{\sum_{j=1}^{n}}\mu_j^{(0)} x_j)}{\mu_k^{(k)} x_k}}-1\right)}{-\lambda_k\displaystyle{\prod_{\substack{l=1 \\ l\neq k}}^{n}} (\mu_k^{(k)} x_k \lambda_{l}-\mu_l^{(l)} x_l \lambda_{k})}  \right ] 
	\end{aligned}
\end{align*}
Finally the constraint \eqref{eq17} reduces to,
\begin{align}
	\begin{aligned}	\label{eq18}
		& \prod_{i=1}^{n} \lambda_{i}  \left [\sum_{k=1}^{n} \dfrac{(\mu_k^{(k)} x_k)^{n-1}\left (e^{\frac{-\lambda_{k}(\tau-{\sum_{j=1}^{n}}\mu_j^{(0)} x_j)}{\mu_k^{(k)} x_k}}-1\right )}{-\lambda_k\displaystyle{\prod_{\substack{l=1 \\ l\neq k}}^{n}} (\mu_k^{(k)} x_k \lambda_{l}-\mu_l^{(l)} x_l \lambda_{k})}  \right ] \leq 1- \beta
	\end{aligned}
\end{align}
Now we can replace the chance constraint in \eqref{eq4} with the deterministic constraint \eqref{eq18} to obtain the robust counterpart of the problem.\\
Therefore robust counterpart of the chance constrained problem \eqref{eq4} can be written in a nonlinear problem as,
\begin{align}
\begin{aligned}	\label{eq19}
	&\min_{x_i} && \frac{1}{2} \sum_{i=1}^{n} \sum_{j=1}^{n} \sigma_{ij} x_i x_j \\
	&\textrm{s.t.:} && \prod_{i=1}^{n} \lambda_{i}  \left [\sum_{k=1}^{n} \dfrac{(\mu_k^{(k)} x_k)^{n-1}\left (e^{\frac{-\lambda_{k}(\tau-{\sum_{j=1}^{n}}\mu_j^{(0)} x_j)}{\mu_k^{(k)} x_k}}-1\right )}{-\lambda_k\displaystyle{\prod_{\substack{l=1 \\ l\neq k}}^{n}} (\mu_k^{(k)} x_k \lambda_{l}-\mu_l^{(l)} x_l \lambda_{k})}  \right ] \leq 1- \beta,\\
	& && \sum_{i=1}^{n} {x_i}=1, \quad {x_i}\geq 0, \ i=1, 2, \dots, n,
\end{aligned}
\end{align}
which summarizes the following result:
\begin{theorem} \label{theorem3}
Consider the chance constrained portfolio optimization problem \eqref{eq4}. Let the uncertain parameters $\zeta_j, \ j=1, 2, \dots, n$ are independent exponential random variables with means $\dfrac{1}{\lambda_j}, \ j=1, 2, \dots, n$. Then the deterministic form of the robust counterpart of the problem \eqref{eq4} is given by a nonlinear problem,
\begin{align*}
	\begin{aligned}	
		&\min_{x_i} && \frac{1}{2} \sum_{i=1}^{n} \sum_{j=1}^{n} \sigma_{ij} x_i x_j \\
		&\textrm{s.t.:} && \prod_{i=1}^{n} \lambda_{i}  \left [\sum_{k=1}^{n} \dfrac{(\mu_k^{(k)} x_k)^{n-1}\left (e^{\frac{-\lambda_{k}(\tau-{\sum_{j=1}^{n}}\mu_j^{(0)} x_j)}{\mu_k^{(k)} x_k}}-1\right )}{-\lambda_k\displaystyle{\prod_{\substack{l=1 \\ l\neq k}}^{n}} (\mu_k^{(k)} x_k \lambda_{l}-\mu_l^{(l)} x_l \lambda_{k})}  \right ] \leq 1- \beta,\\
		& && \sum_{i=1}^{n} {x_i}=1, \quad {x_i}\geq 0, \ i=1, 2, \dots, n.
	\end{aligned}
\end{align*}
\end{theorem}
\section{Numerical Example} \label{sec4}
\subsection{Data Description}
For our example, we consider the stock market data of three major sector indices of India: (i) Nifty Bank (Banking Sector), (ii) Nifty Infra (Infrastructure Sector), (iii) Nifty IT (Information Technology Sector). The stock price data of these sectors are collected from \href{https://finance.yahoo.com}{https://finance.yahoo.com} for the period June 2017 to May 2022. These stock price values are illustrated in Fig. \ref{fig1}.
\begin{figure}[h!]
\centering 
	\includegraphics[height=7.5 cm]{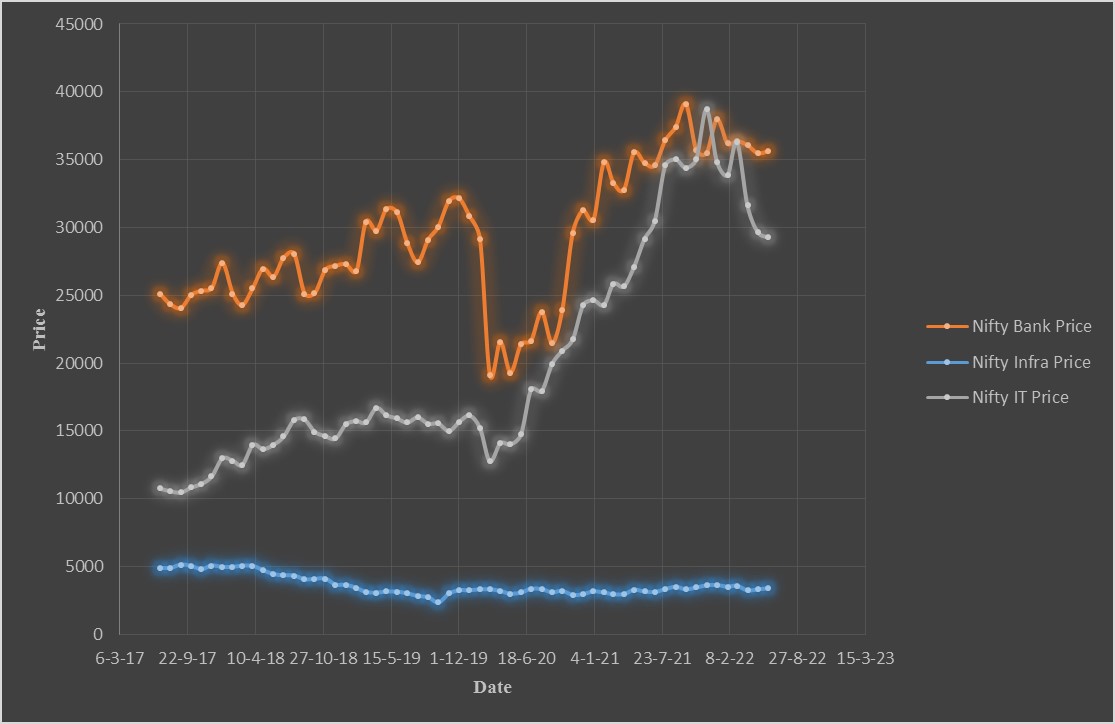}
	\caption{Stock Price Data of Three Sectors: Nifty Bank, Nifty Infra and Nifty IT from June 2017 to May 2022}
	 \label{fig1}
\end{figure}
\\
It can be observed that Nifty IT gives the highest return during this period, whereas Nifty Infra provides the least return. However, Nifty Infra is the least volatile sector, so it is the least risky and Nifty Bank and Nifty IT are the more risky sectors during this period. Our objective is to use this data to determine the optimal allocation among these three sectors to achieve a target portfolio return with a minimum portfolio risk.
For this, we first calculate the expected quarterly returns and the covariance of returns of the three sectors, which are to be used as the input parameters for our portfolio problem. Considering the uncertainty into account, we term the "input parameters" as the "nominal values of the input parameters", which are given in Table \ref{table1}.
\begin{table}[h!]
	\centering
	\caption{Nominal Values of the Input Parameters}
	\begin{tabular}{| c | c | c  c  c |}
		\hline
		\cline{1-5}
		\textbf{Sectors} &  {\textbf{Expected Returns (in $\%$)}} & \multicolumn{3}{c |} {\textbf{Covariance of Returns (in $\%$)}} \\
		\cline{3-5}
		{} & {} & Nifty Bank & Nifty Infra & Nifty IT  \\
		\hline
	 Nifty Bank & $2.609$ & $24.126 $ & $-1.460 $ & $11.032$  \\
		
		 Nifty Infra & $-1.430$ & $-1.460 $ & $8.237$ & $0.461$   \\
		
	 Nifty IT & $6.329 $ & $11.032$ & $0.461$ & $18.034$   \\
	
		\hline
	\end{tabular}
\label{table1}
\end{table}
\subsection{Nominal Problem}
First, we formulate and solve the portfolio problem without considering the uncertainty. From Table \ref{table1}, the nominal expected return vector is given by,
\begin{align*}
	\bm{{\bm{\mu}^{(0)}}} = \begin{bmatrix} 
		\mu_1^{(0)} \\
	\mu_2^{(0)}\\
		\mu_3^{(0)} \\
	\end{bmatrix}
= \begin{bmatrix} 
2.609 \\
-1.430 \\
6.329 \\
\end{bmatrix}
\end{align*}
Then the nominal portfolio problem is formulated as,
\begin{align} 
\begin{aligned} \label{eq21}	
&\min_{x_i} && \frac{1}{2} \left[24.126x_1^2+8.237x_2^2+18.034x_3^2+ 2 \cdot (-1.460) x_1x_2+2 \cdot 11.032 x_1x_3+ 2 \cdot 0.461 x_2x_3 \right]\\
& s.t.: && 2.609x_1-1.430x_2+6.329x_3 \geq \tau,\\
& &&x_1+x_2+x_3=1, \quad x_1, x_2, x_3 \geq 0 
\end{aligned}
\end{align}
Since the portfolio is a combination of all the individual assets, there is a high possibility that the portfolio return lies near the average of individual assets returns. Therefore we assume our target return ($\tau$) as some values within the range from $1.5$ to $3.5$. We solve this problem for these levels of $\tau$ to get the optimal allocations and consequently calculate the corresponding optimal portfolio risks. These optimal solutions for our nominal problem are given in Table \ref{table2}.
\begin{table}[h!]
\centering
\caption{Optimal Solutions for Nominal Problem}
\begin{tabular}{| c | c  c  c | c |}
\hline
\cline{1-5}
\textbf{Target Return($\tau$)} &  \multicolumn{3}{c |} {\textbf{Optimal Allocation}} &  {\textbf{Optimal Portfolio Risk}} \\
\cline{2-4}
{}  & Nifty Bank & Nifty Infra & Nifty IT & {} \\
\hline
$1.5$ & $0.1415$ & $0.5545$ & $0.3040 $ & $2.7788$  \\

$1.7$ & $0.1328$ & $0.5329$ & $0.3343$ & $2.8585$   \\

$1.9$ & $0.1242$ & $0.5113$ & $0.3645$ & $2.9535$   \\

$2.1$ & $0.1155$ & $0.4897$ & $0.3948 $ & $3.0638$  \\

$2.3$ & $0.1068$ & $0.4680$ & $0.4251 $ & $3.1893$   \\

$2.5$ & $0.0982$ & $0.4464$ & $0.4554$ & $3.3301$   \\

$2.7$ & $0.0895$ & $0.4248$ & $0.4857$ & $3.4861$  \\

$2.9$ & $0.0809$ & $0.4032$ & $0.5160 $ & $3.6574$  \\

$3.1$ & $0.0722$ & $0.3815$ & $0.5463$ & $3.8440$   \\

$3.3$ & $0.0635$ & $0.3599$ & $0.5765 $ & $4.0459$  \\

$3.5$ & $0.0549$ & $0.3383$ & $0.6068$ & $4.2630$   \\

\hline
\end{tabular}
\label{table2}
\end{table}
\\
Using these results, we plot the optimal allocation graph for the nominal problem in Fig. \ref{fig2}, where the black portion indicates the weight of the Bank Nifty sector, red indicates the Nifty Infra's weight, and blue indicates Nifty IT's weight in the optimal portfolio. Moreover, from different levels of target returns and their corresponding optimal risks, we plot the efficient frontier, which is illustrated in Fig. \ref{fig3}. The efficient frontier is a risk-return tradeoff curve, which plots the optimal risk values at different levels of returns or vice-versa.
\begin{figure}[h!]
\centering
	\includegraphics[height=8 cm]{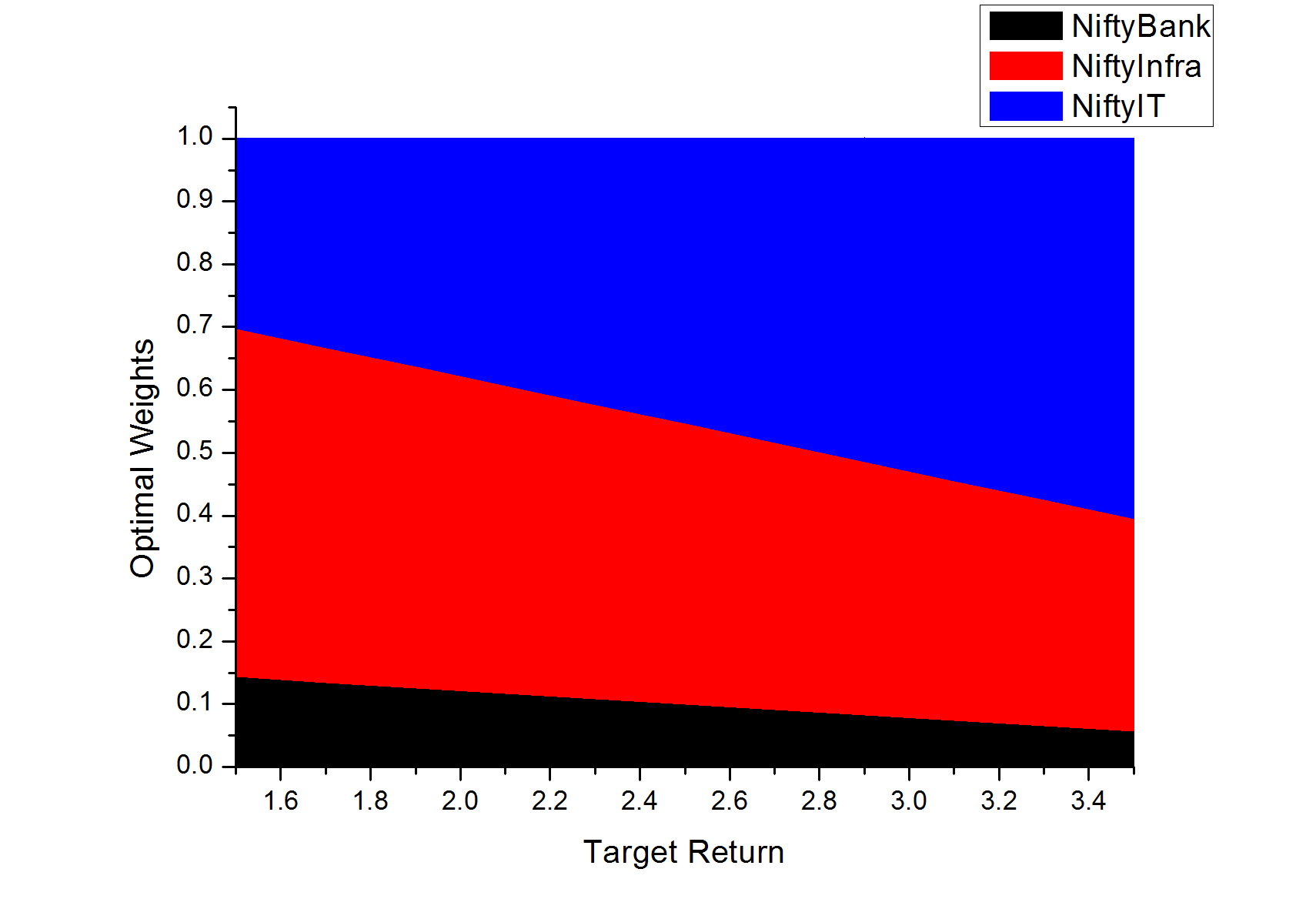}
	\caption{Optimal Allocation for Different Return Levels of the Nominal Problem}
\label{fig2}
\end{figure}
\begin{figure}[h]
\centering
	\includegraphics[height=8 cm]{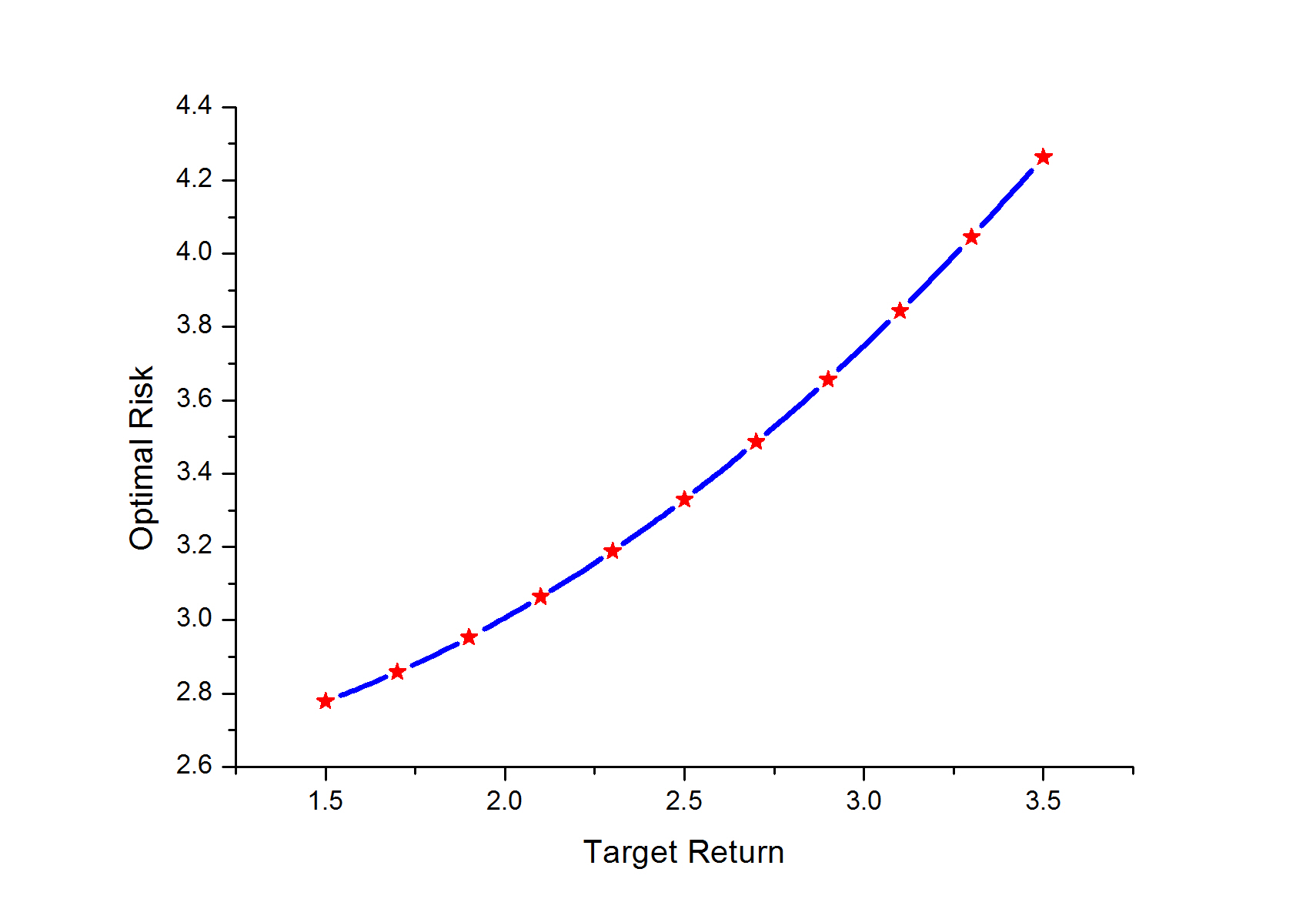}
	\caption{Efficient Frontier of the Portfolio for the Nominal Problem}
\label{fig3}
\end{figure}
\subsection{Chance Constraint Based Problem}
Now taking uncertainty into account, we convert the portfolio problem into a chance constraint based optimization problem. First, the quantification of uncertainty in expected returns are assumed as follows:
Let $\zeta_1$, $\zeta_2$, $\zeta_3$ be the perturbations associated with the expected return vector, and the basic shifts of the perturbations are given by,
\begin{align*}
	\bm{{\bm{\mu}^{(1)}}} = \begin{bmatrix} 
		\mu_1^{(1)} \\
		\mu_2^{(1)}\\
		\mu_3^{(1)} \\
	\end{bmatrix}
	= \begin{bmatrix} 
		0.2 \\
		0 \\
		0 \\
	\end{bmatrix}, 
\quad 
\bm{{\bm{\mu}^{(2)}}} = \begin{bmatrix} 
	\mu_1^{(2)} \\
	\mu_2^{(2)}\\
	\mu_3^{(2)} \\
\end{bmatrix}
= \begin{bmatrix} 
	0 \\
	0.1 \\
	0 \\
\end{bmatrix},
\quad
\bm{{\bm{\mu}^{(3)}}} = \begin{bmatrix} 
	\mu_1^{(3)} \\
	\mu_2^{(3)}\\
	\mu_3^{(3)} \\
\end{bmatrix}
= \begin{bmatrix} 
	0 \\
	0 \\
	0.3 \\
\end{bmatrix}.
\end{align*}
Although the level of confidence $\beta$ of the chance constraint can be any value in the interval $(0,1)$, it should be close to 1 for the solution to be more accurate. Thus, we set the value of $\beta$ as $0.95$ for our problem.\\
Then the chance constrained portfolio problem is obtained as,
\begin{align} 
\begin{aligned}	\label{eq20}
	&\min_{x_i} && \frac{1}{2} \left[24.126x_1^2+8.237x_2^2+18.034x_3^2+ 2 \cdot (-1.460) x_1x_2+2 \cdot 11.032 x_1x_3+ 2 \cdot 0.461 x_2x_3 \right]\\
	& s.t.: && {Prob}_{\bm{\zeta_j}} \left \{2.609x_1-1.430x_2+6.329x_3+0.2x_1 \zeta_1+0.1 x_2 \zeta_2+ 0.3x_3 \zeta_3\geq \tau \right \} \geq 0.95\\
	&  && x_1+x_2+x_3=1, \quad x_1, x_2, x_3 \geq 0
\end{aligned}
\end{align}
By using the results in Theorems \ref{theorem1} and \ref{theorem3}, we obtain the robust counterparts of chance constrained problem \eqref{eq20} for the two cases-- (i) when the perturbations follow the normal distribution, (ii) when the perturbations follow the exponential distribution and solving these robust counterparts we get the robust solutions.
\subsubsection{When the Perturbations Follow Normal Distribution}
Suppose the perturbations $\zeta_1$, $\zeta_2$, $\zeta_3$ follow normal distribution, each with mean $0$ and standard deviation $1$.\\
Then by using Theorem \ref{theorem1}, the deterministic robust counterpart of our chance constrained problem \eqref{eq20} is given by,
\begin{align} 
\begin{aligned} \label{eq22}	
&\min_{x_i} &&  \frac{1}{2} \left[24.126x_1^2+8.237x_2^2+18.034x_3^2+ 2 \cdot (-1.460) x_1x_2+2 \cdot 11.032 x_1x_3+ 2 \cdot 0.461 x_2x_3 \right]\\
& s.t.: && 2.609x_1-1.430x_2+6.329x_3+\sqrt{2} \cdot erf^{-1} (-0.9) \cdot \sqrt {0.04 x_1^2+0.01 x_2^2+0.09 x_3^2} \geq \tau\\
&  && x_1+x_2+x_3=1, \quad x_1, x_2, x_3 \geq 0 \\
\end{aligned}
\end{align}
We solve this problem for different levels of target return ($\tau$) and calculate the corresponding optimal portfolio risks. The problem is solved using MATLAB programming, and the optimal results for the normally distributed perturbation case are given in Table \ref{table3}.
\begin{table}[h!]
	\centering
	\caption{Optimal Solutions When the Perturbations Follow Normal Distribution}
	\begin{tabular}{| c | c  c  c | c |}
		\hline
		\cline{1-5}
		\textbf{Target Return($\tau$)} &  \multicolumn{3}{c |} {\textbf{Optimal Allocation}} &  {\textbf{Optimal Portfolio Risk}} \\
		\cline{2-4}
		{}  & Nifty Bank & Nifty Infra & Nifty IT & {} \\
		\hline
		$1.5$ & $0.1371$ & $0.5321$ & $0.3308 $ & $2.8546$  \\
		
		$1.7$ & $0.1290$ & $0.5087$ & $0.3623$ & $2.9547$   \\
		
		$1.9$ & $0.1210$ & $0.4853$ & $0.3938$ & $3.0723$   \\
		
		$2.1$ & $0.1129$ & $0.4617$ & $0.4253 $ & $3.2075$  \\
		
		$2.3$ & $0.1049$ & $0.4382$ & $0.4570 $ & $3.3602$  \\
		
		$2.5$ & $0.0968$ & $0.4145$ & $0.4887$ & $3.5307$   \\
		
		$2.7$ & $0.0888$ & $0.3908$ & $0.5204$ & $3.7188$  \\
		
		$2.9$ & $0.0807$ & $0.3671$ & $0.5522 $ & $3.9248$   \\
		
		$3.1$ & $0.0727$ & $0.3433$ & $0.5840$ & $4.1485$   \\
		
		$3.3$ & $0.0646$ & $0.3196$ & $0.6158 $ & $4.3901$  \\
		
		$3.5$ & $0.0566$ & $0.2958$ & $0.6477$ & $4.6494$   \\
		
		\hline
	\end{tabular}
\label{table3}
\end{table}
\\
Using these results, we plot the optimal allocation graph and the efficient frontier for the problem with the normally distributed perturbation, respectively given in Fig. \ref{fig4} and Fig. \ref{fig5}.
\begin{figure}[h!]
\centering
	\includegraphics[height=8 cm]{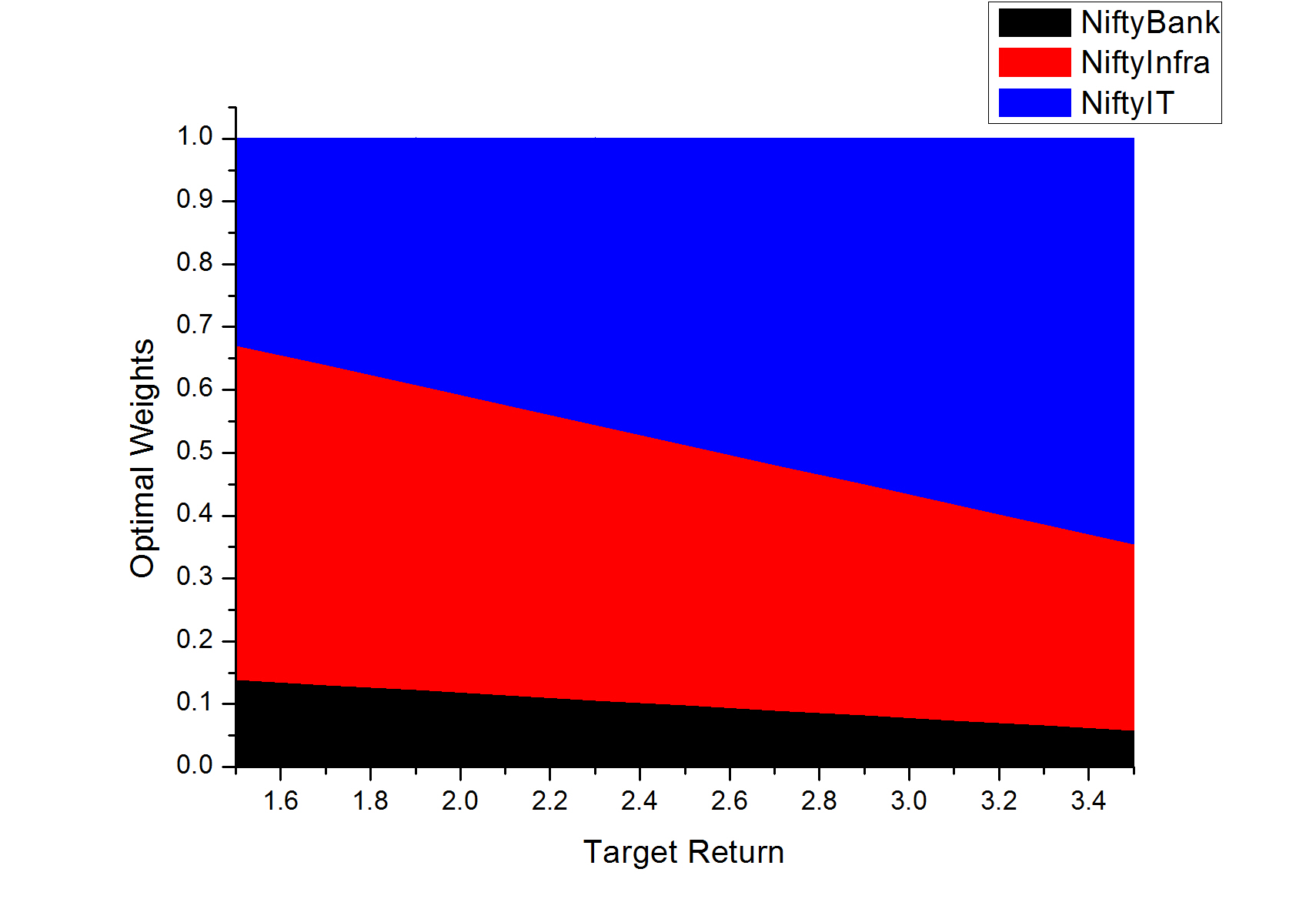}
	\caption{Optimal Allocation for Different Return Levels When the Perturbations Follow Normal Distribution}
\label{fig4}
\end{figure}
\begin{figure}[h!]
\centering
	\includegraphics[height=8 cm]{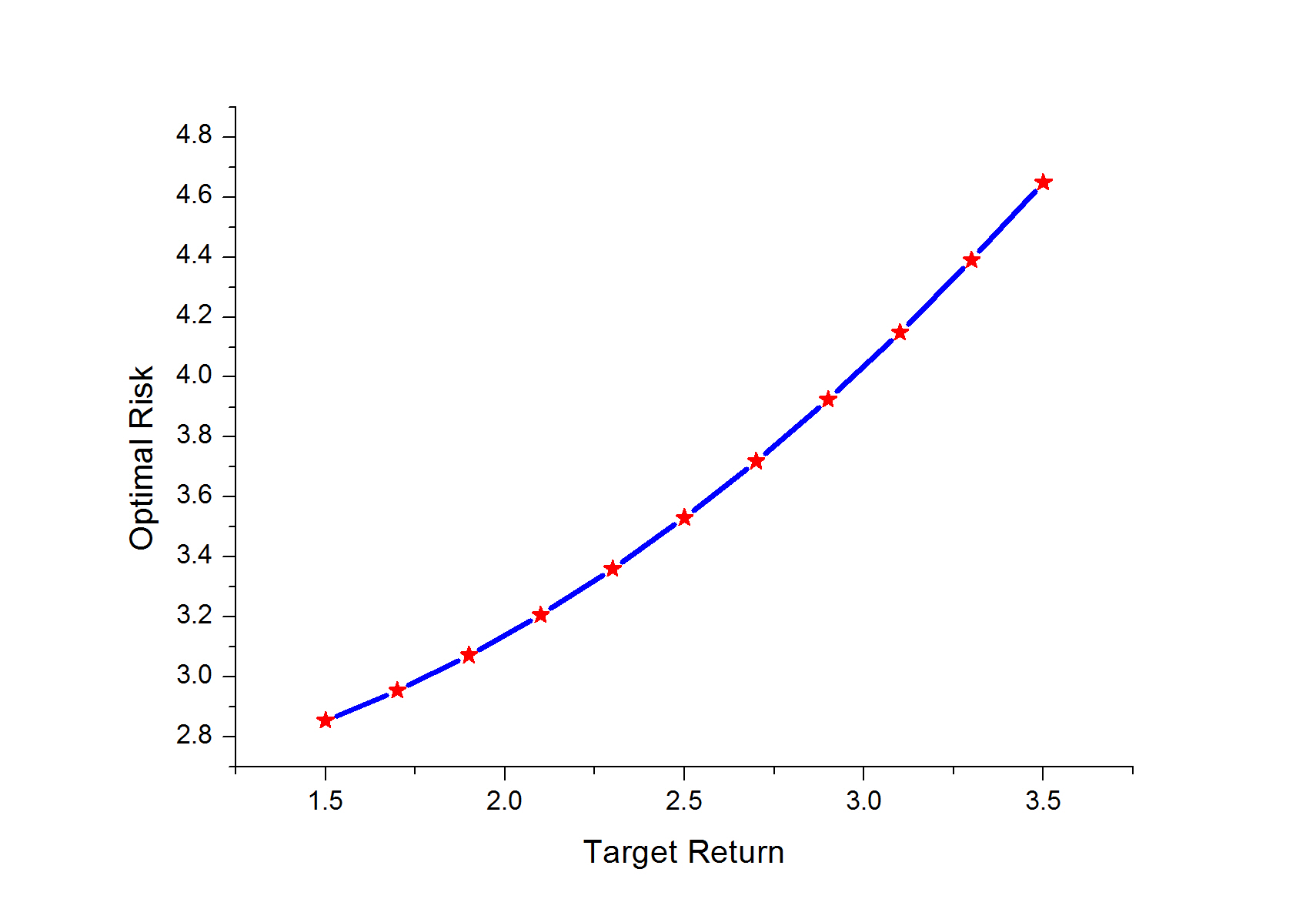}
	\caption{Efficient Frontier of the Portfolio When the Perturbations Follow Normal Distribution}
\label{fig5}
\end{figure}
\subsubsection{When the Perturbations Follow Exponential Distribution}
Suppose the perturbations $\zeta_1$, $\zeta_2$, $\zeta_3$ follow exponential distribution each with mean $1$. So we have, $\lambda_1=\lambda_2=\lambda_3=1$.\\
Then by using Theorem \ref{theorem3} the deterministic robust counterpart of our chance constrained problem \eqref{eq20} is written as,
\begin{equation}
\begin{aligned}	\label{eq23}
\begin{split}
&\min_{x_i} && \frac{1}{2} \left[24.126x_1^2+8.237x_2^2+18.034x_3^2+ 2 \cdot (-1.460) x_1x_2+2 \cdot 11.032 x_1x_3+ 2 \cdot 0.461 x_2x_3 \right]\\
& s.t.: && (1) \cdot \Biggl[ \dfrac{0.04 \cdot x_1^2 \cdot \left (e^{\frac{-[\tau-(2.609x_1-1.430x_2+6.329x_3)]}{0.2 x_1}}-1\right )}{(-1)[(0.2)x_1(1)-(0.1)x_2(1)]\cdot[(0.2)x_1(1)-(0.3)x_3(1)]} \\
& && \quad \qquad +\dfrac{0.01 x_2^2 \cdot \left (e^{\frac{-[\tau-(2.609x_1-1.430x_2+6.329x_3)]}{0.1 x_2}}-1 \right )}{(-1)[(0.1)x_2(1)-(0.2)x_1(1)] \cdot [(0.1)x_2(1)-(0.3)x_3(1)]} \\
& && \quad \qquad +\dfrac{0.09 x_3^2 \cdot \left ( e^{\frac{-[\tau-(2.609x_1-1.430x_2+6.329x_3)]}{0.3 x_3}}-1 \right )}{(-1)[(0.3)x_3(1)-(0.2)x_1(1)] \cdot [(0.3)x_3(1)-(0.1)x_2(1)]} \Biggr] \leq 0.05\\
&  && x_1+x_2+x_3=1, \quad x_1, x_2, x_3 \geq 0 \\
\end{split}
\end{aligned}
\end{equation}
We solve this problem for different levels of target return ($\tau$) and calculate the corresponding optimal portfolio risks. The problem is solved using MATLAB programming and the optimal results for the exponential distribution case are given in Table \ref{table4}.
\begin{table}[ht]
	\centering
	\caption{Optimal Solutions When the Perturbations Follow Exponential Distribution}
	\begin{tabular}{| c | c  c  c | c |}
		\hline
		\cline{1-5}
		\textbf{Target Return($\tau$)} &  \multicolumn{3}{c |} {\textbf{Optimal Allocation}} &  {\textbf{Optimal Portfolio Risk}} \\
		\cline{2-4}
		{}  & Nifty Bank & Nifty Infra & Nifty IT & {} \\
		\hline
		$1.5$ & $0.0001$ & $0.6216$ & $0.3783 $ & $2.9906$  \\
		
		$1.7$ & $0.0000$ & $0.5963$ & $0.4036$ & $3.0447$   \\
		
		$1.9$ & $0.0746$ & $0.4818$ & $0.4436$ & $3.2087$   \\
		
		$2.1$ & $0.0000$ & $0.5450$ & $0.4550 $ & $3.2045$  \\
		
		$2.3$ & $0.0083$ & $0.4985$ & $0.4932$ & $3.3701$   \\
		
		$2.5$ & $0.0000$ & $0.4933$ & $0.5067$ & $3.4324$   \\
		
		$2.7$ & $0.0003$ & $0.4632$ & $0.5366$ & $3.5955$  \\
		
		$2.9$ & $0.0006$ & $0.4411$ & $0.5583 $ & $3.7287$   \\
		
		$3.1$ & $0.0782$ & $0.3840$ & $0.5378$ & $3.8042$   \\
		
		$3.3$ & $0.0001$ & $0.3758$ & $0.6240 $ & $4.2021$  \\
		
		$3.5$ & $0.0000$ & $0.3388$ & $0.6612$ & $4.5184$   \\
		
		\hline
	\end{tabular}
\label{table4}
\end{table}
Using these results, we plot the optimal allocation graph and the efficient frontier for the problem with exponentially distributed perturbations, which are respectively given in Fig. \ref{fig6} and Fig. \ref{fig7}.
\begin{figure}[h!]
\centering 
\includegraphics[height=8 cm]{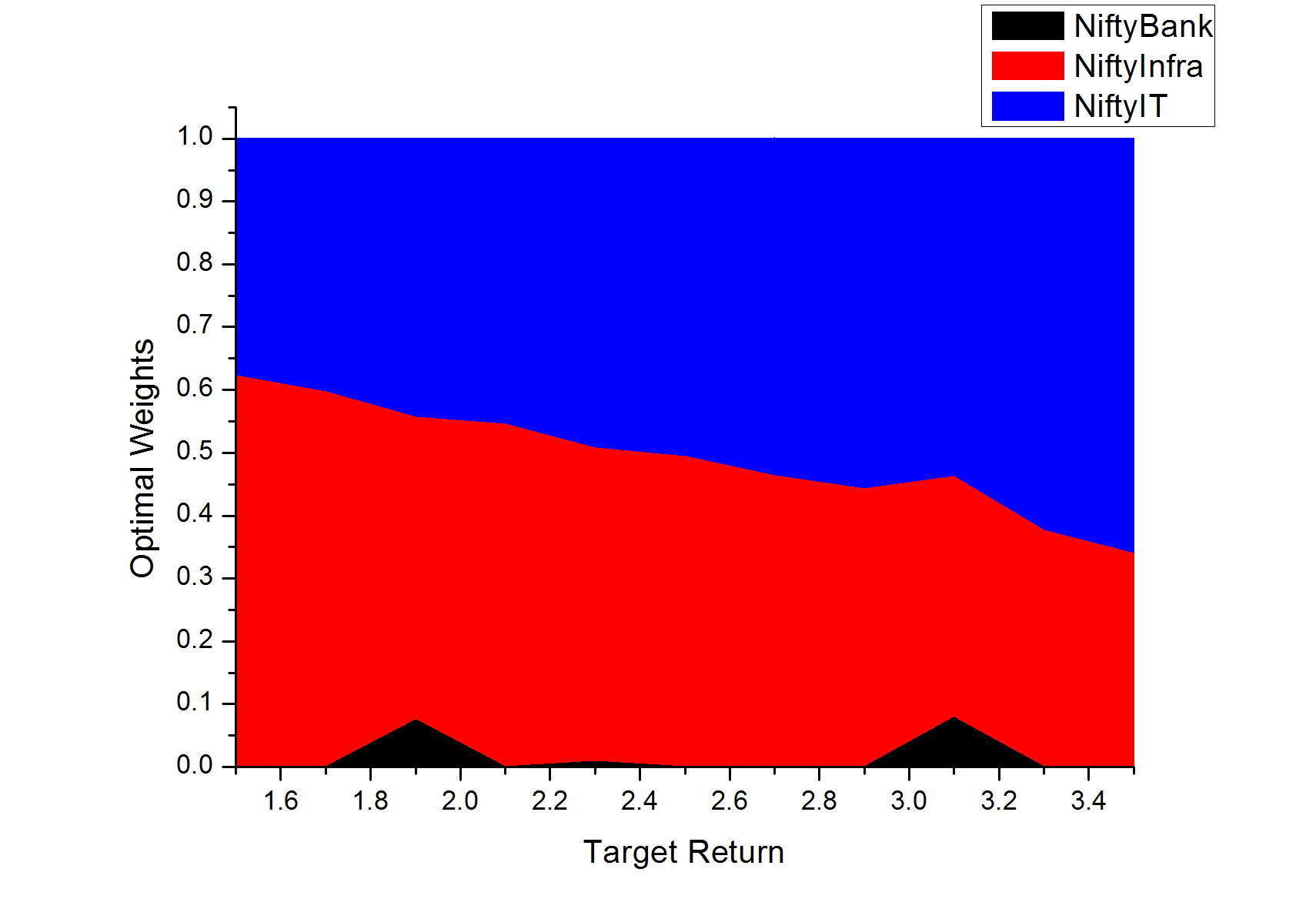}
\caption{Optimal Allocation for Different Return Levels When the Perturbations Follow Exponential Distribution}
\label{fig6}
\end{figure}
\begin{figure}[h!]
\centering 
\includegraphics[height=8 cm]{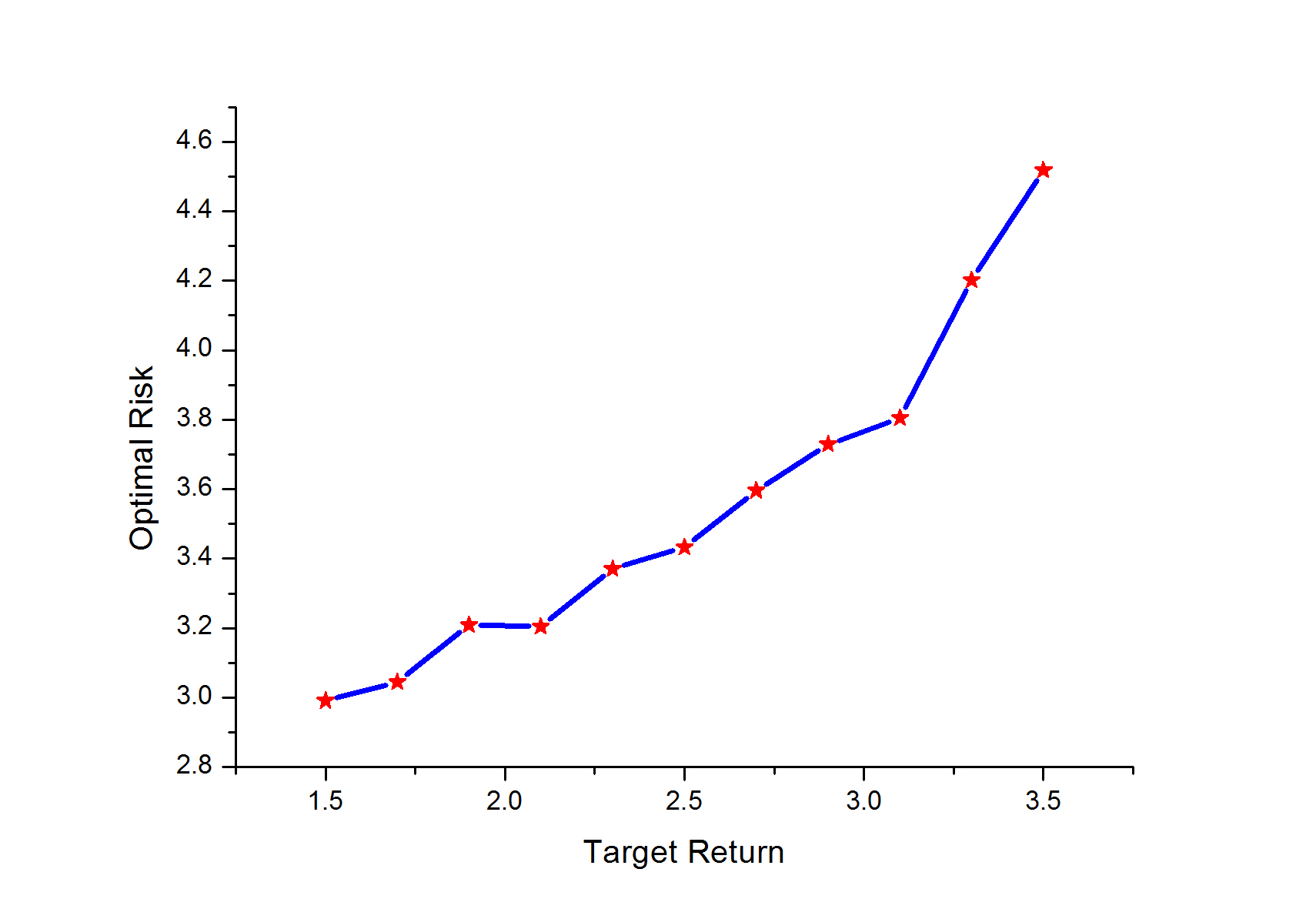}
\caption{Efficient Frontier of the Portfolio When the Perturbations Follow Exponential Distribution}
\label{fig7}
\end{figure}
\subsection{Dissimilarity Analysis of the Models}
To check how much the robust portfolio models under normal and exponential distributions differ from the nominal portfolio model, first we obtain the dissimilarity matrix $(D)$ of these models. The entry $D_{pq}$ $(p, q=1, 2, 3)$ of matrix $D$ is the distance between $p^{th}$ and $q^{th}$ models. It is clear that $D$ is a symmetric matrix with zeros as the diagonal elements. In our problem, distances among the three models are calculated from the optimal portfolio risk results obtained in Tables \ref{table2}, \ref{table3}, and \ref{table4}. Since the associated data for finding the distances (i.e., the optimal portfolio risks) are of numeric type, we can use the Euclidean norm as a distance measure. For instance, let us denote the indices of nominal, robust normal and robust exponential models as $1, 2, 3$ respectively. Then distance between the results of nominal and robust normal model is calculated as,
\begin{align*} 
\begin{aligned}
& D_{12} \text{ or } D_{21}= \sqrt{(2.7788-2.8546)^2+(2.8585-2.9547)^2+\dots+(4.2630-4.6494)^2}=0.8298
\end{aligned}
\end{align*}
Consequently the dissimilarity matrix among the three models is obtained as,
\begin{align}
D = \left(
\begin{array}{c|c | c | c}
 & \textbf{Nominal}  & \textbf{Robust Normal} & \textbf{Robust Exponential}\\
\hline
\textbf{Nominal} & 0 & 0.8298 & 0.5621\\
\textbf{Robust Normal} & 0.8298 & 0 & 0.5195\\
\textbf{Robust Exponential} & 0.5621 & 0.5195 & 0
\end{array}
\right)
\end{align}
The more is the distance between the results of two models; the more is the dissimilarity between them. Thus it is evident from the dissimilarity matrix that results of the robust model with normally distributed perturbations are very much dissimilar from results of the nominal model. This shows that under the uncertain parameters as defined for our numerical example, the normally distributed perturbations can affect the optimal solutions more than the exponentially distributed perturbations. On the other hand, the results of the two robust models are the least dissimilar.

\section{Conclusion} \label{sec5}
This paper discusses the chance constrained portfolio optimization problems. The main objective is to find out the deterministic robust counterpart of such chance constrained problems. In particular, the robust counterpart of uncertain portfolio optimization model is derived when the perturbations follow normal and exponential distributions. The robust counterpart model obtained for the normally distributed perturbations is given by an SOCP problem, whereas the robust model for the exponentially distributed perturbations is given by a nonlinear programming problem. It is found that when the perturbations are normally distributed, each having zero mean, the robust counterpart of the chance constrained portfolio model for confidence level $\beta=0.5$ is the same as the nominal portfolio model. Moreover, the robust counterpart results are used to solve an Indian stock market problem, and the results are analyzed.\\
Since the portfolio model that is considered in the paper is a quadratic programming problem, the results can be used to obtain the robust counterparts of any chance constrained quadratic programming with normal and exponential perturbations. In fact, the deterministic robust counterpart of any chance constrained problem with linear chance constraint(s) can be derived analogously. There is a scope for further research in finding out the robust counterparts of uncertain problems involving nonlinear chance constraints.
\section*{Declarations}
\subsection*{Data Availability}
The Nifty price data of the three sectors– Nifty Bank, Nifty Infra and Nifty IT
used in the study are collected for the time period June 2017 to May 2022 and the data is available in the website \href{https://finance.yahoo.com}{https://finance.yahoo.com}. We also attach the data in the MS Excel file named "Nifty price data.xls".
\subsection*{Conflict of Interest} The authors declare that they have no conflicts of interest.

%
%



\end{document}